\documentclass{birkmult}

\usepackage{amssymb}

\usepackage{amscd}
\newtheorem{Thm}{Theorem}[section]
 
 \newtheorem{Lem}[Thm]{Lemma}
 \newtheorem{Prop}[Thm]{Proposition}
 \theoremstyle{definition}
 
 \theoremstyle{remark}
 \newtheorem{Rem}[Thm]{Remark}
 \newtheorem{Ex}[Thm]{Example}
 \numberwithin{equation}{section}
\newtheorem{Ques}[Thm]{Question}

\begin{document}

\title[K\"ulshammer ideals and dihedral blocks]
{K\"ulshammer ideals and the scalar
problem for blocks with dihedral defect groups}






\author{Thorsten HOLM}

\address{
Institut f\"{u}r Algebra, Zahlentheorie und Diskrete Mathematik\\
Fakult\"{a}t f\"{u}r Mathematik und Physik\\
 Leibniz Universit\"{a}t Hannover\\
Welfengarten 1 \\D-30167 
Hannover\\ Germany}
\email{holm@math.uni-hannover.de}

\author{Guodong ZHOU}

\address{
Institut de Math\'ematiques de Jussieu\\
\'Equipe  Th\'eorie des Groupes\\
 Case 7012, 2 place Jussieu \\
F-75251 Paris Cedex 05\\
France \\
Current address:\\ 
Universit\"{a}t zu K\"{o}ln\\
 Mathematisches Institut\\
  Weyertal 86-90 \\D-50931
K\"{o}ln \\Germany }
 \email{gzhou@math.uni-koeln.de}

\thanks{The work presented in this paper has been supported by a
German-French grant DAAD-PROCOPE, DAAD grant number D/0628179,
and "partenariat Hubert Curien PROCOPE dossier 14765WB", respectively.
We gratefully acknowledge the financial support.} 

\thanks{Part of this work was done in December
2007 during a visit of the second named author at Leibniz
Universit\"{a}t Hannover. He wishes to express his sincere gratitude
to Prof. Christine Bessenrodt and the first named author for their
invitation and generous hospitality. During the preparation of
this work, the second named author benefited from financial support via
postdoctoral fellowships from the Ecole Polytechnique and 
from the network "Representation theory of algebras and
algebraic Lie theory". He would like to thank Prof. Bernhard Keller
and Prof. Steffen K\"{o}nig for their support.}

\begin{abstract}  In by now classical work, K.~Erdmann 
classified blocks of finite groups with
dihedral defect groups (and more generally algebras of dihedral type)
up to Morita equivalence. In the explicit description by quivers
and relations of such algebras with two simple modules, several
subtle problems about scalars occurring in relations remained
unresolved. In particular, for the dihedral case it is a longstanding
open question whether blocks of finite groups can occur for both
possible scalars $0$ and $1$.  

In this article, using K\"ulshammer ideals (a.k.a. generalized Reynolds 
ideals), we provide the first examples of blocks where the scalar is 1,
thus answering the above question to the affirmative. Our examples are
the principal blocks of $PGL_2(\mathbb{F}_q)$, the projective general
linear 
group of $2\times 2$-matrices  with entries in the finite field
$\mathbb{F}_q$, where $q=p^n\equiv \pm 1 \ \text{mod}\ 8$ with $p$
an odd prime number.

\end{abstract}

\maketitle





\section{Introduction}

Since the pioneering work of  J.~Rickard
(\cite{R1},\cite{R2},\cite{R3}) and of D.~Happel (\cite{H1},\cite{H2}),
derived equivalences and derived invariants have received much
attention in representation theory. For the representation theory
of finite groups, M.~Brou\'e's abelian defect group
conjecture (\cite{Broue}) plays a most prominent role in these
developments. Although quite a few derived invariants have been
discovered so far, such as the center (\cite{R1}), Hochschild homology 
and cohomology
(\cite{R1}), cyclic homology (\cite{Keller}), $K$-theory (\cite{TT}),
etc., most of these invariants are very difficult to compute. In
the eighties of the last century, B.~K\"{u}lshammer (\cite{K1})
introduced a sequence of ideals in the center of a symmetric algebra
defined over an algebraically closed field of positive
characteristic and he called them \textit{generalized Reynolds
ideals}. L.~H\'{e}thelyi, E.~Horv\'{a}th, B.~K\"{u}lshammer and
J.~Murray proved that these ideals are invariant under Morita equivalence
(\cite{HHKM}). In 2005, A.~Zimmermann (\cite{Z1}) proved
that these ideals are even invariant under derived equivalences. 

A remarkable feature of K\"ulshammer ideals is that they are 
in principal accessible for explicit computations. In particular,
this makes these new
derived invariants potentially useful for distinguishing algebras
up to derived equivalence (which in general is a very hard problem,
due to the lack of 'computable' derived invariants). 

For the definition and more background on K\"ulshammer ideals 
we refer the reader to Section \ref{Sec:Reynolds} below, and
for other recent developments around K\"ulshammer ideals
to the articles \cite{BHZ}, \cite{BHHKM}, \cite{HS}, \cite{HZ},
\cite{Z2}, \cite{Z3}. 

The objective of this article is to present another application of
K\"ulshammer ideals to the scalar problem for blocks with
dihedral defect groups which have two simple modules up to
isomorphisms. Before stating our main result, we review some
background.

Finite-dimensional algebras over an algebraically closed field
are divided
into three (mutually exclusive) representation types: finite, 
tame and wild. For blocks 
of group algebras of finite groups, the representation type is 
characterized by
their defect groups. A block has finite representation type if and
only if its defect groups are cyclic. These blocks are well
understood, see \cite{D66}. Tame representation type only occurs when the
ground field has characteristic 2 and when
the defect groups are dihedral, semi-dihedral or generalized
quaternion. In a series of seminal papers
(\cite{E79}, \cite{E82}, \cite{E87}, \cite{E88}, \cite{E90}, \cite{E882})
and the monograph 
\cite{E1}, K.~Erdmann introduced the larger classes of algebras of dihedral,
semidihedral and quaternion type and 
classified these algebras up to Morita
equivalence. Based on her wor, the first named author later 
classified these algebras up to derived equivalence (\cite{Holm1}, \cite{Holm2}). 
Nevertheless, several subtle problems remain
unresolved, most of them connected to certain scalars occurring
in relations. Let us explain in detail the situation for dihedral
blocks with two simple modules. We recall the classification of
K.~Erdmann and the first named author.

\begin{Thm}(\cite{E1}, \cite[Proposition (2.1)]{Holm1}) 
Let $k$ be an algebraically closed field of
characteristic 2 and $B$ a dihedral block of a finite group with
dihedral defect groups, of order $2^n$, and with two simple modules.
Then there exists a scalar $c\in \{0, 1\}$ such that $B$ is derived
equivalent to the algebra $D(2\mathcal{A})^{s}(c)$ with  $s=2^{n-2}$
defined by the following quiver with relations:

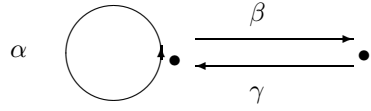
\begin{center}
\unitlength0.6cm
\begin{picture}(12,4)
  \put(5,1.7){$\bullet$}
\put(9.2,1.8){$\bullet$} \put(5.6,2.3){\vector(1,0){3.5}}
  \put(3.8,2){\circle{2.0}}
 \put(4.85,1.85){\vector(0,1){0.3}}
\put(9.1,1.7){\vector(-1,0){3.5}} \put(1.5,1.9){$\alpha$}
\put(6.8,2.7){$\beta$} \put(6.8,1.0){$\gamma$}
\end{picture}

$ \gamma \beta=0, \ \alpha^2=c(\alpha\beta\gamma)^s,$ 
$(\alpha\beta\gamma)^s=(\beta\gamma\alpha)^s$ \end{center}
\end{Thm}

The following is a longstanding open problem. 
\begin{Ques} Can both values $c=0$ and $c=1$ occur for blocks of finite
groups with dihedral defect groups and with two simple modules?
\end{Ques}

It is known that the above algebras for different scalars $c=0$ and $c=1$
are not derived equivalent (\cite{Kauer}, see also \cite{HZ}). 

The following seems to be a complete list of dihedral blocks 
where the scalar could so far be determined; note that it is usually 
very hard to determine the value of the scalar directly, even for small 
examples. 
 
\begin{Ex} Let $k$ be an algebraically closed field of
characteristic $2$.

(1) The group algebra $kS_4$ of the symmetric group $S_4$
is a block with dihedral defect groups
and has two simple modules; for this block, the scalar is $c=0$ 
(\cite{E1}).

(2) In \cite[Section 1.5]{E2}, K.~Erdmann constructed infinitely 
many dihedral blocks with two simple modules for which $c=0$.
These blocks are principal blocks of certain quotients of $G=GU_2(q)$,
the general unitary group, where $q\equiv 3\mod 4$. 
More precisely,
let $\overline{G}:= G/O_{2'}(G)$, then the principal $2$-blocks
of $G_1:=\overline{G}/Z(\overline{G})$ have dihedral defect groups
and two simple module. Note that these blocks are not Morita
equivalent to an algebra of the form $D(2\mathcal{A})^{s}(0)$
(the quivers of their basic algebras have two loops, see 
\cite[Section 1.5]{E2}), but they are derived equivalent
to some algebra $D(2\mathcal{A})^{s}(0)$ by 
\cite[Proposition (2.1)]{Holm1}. 

\end{Ex}

In view of these examples, it was believed by experts
that the scalar $c$
should be always $0$ for blocks of finite groups. 
Surprisingly, we prove in this article the
following theorem which gives the first examples for which the
scalar $c=1$ does occur.

\begin{Thm} \label{main} Let  $k$ be an algebraically
closed field of characteristic $2$. Suppose that $q=p^m\equiv \pm 1\
\text{mod}\ 8$ for $p$ an odd prime number. Then the principal block
of the group algebra $kPGL_2(q)$ of the projective general linear
group is a dihedral block with two simple modules for which
$c=1$.
\end{Thm}

As a direct application we get that the dihedral 
blocks with two simple modules considered in \cite{E2} can not 
be derived equivalent
to the principal blocks of the projective general linear groups
considered here. 
\smallskip

The main tool of the proof of our main theorem
are K\"ulshammer ideals, a.k.a.
generalized Reynolds ideals. These form a descending series
of ideals of the center of a symmetric algebra in positive
characteristic; for the definition and basic properties we
refer to Section \ref{Sec:Reynolds} below. 
The crucial 
fact we use is a recent result of A.\,Zimmermann and the first
author \cite[Thms 1.1 and 4.1]{HZ}, showing that for different
scalars $c=0$ and $c=1$, the factor rings of the centre of the
block modulo the first K\"ulshammer ideal are not isomorphic. 
More precisely, one can distinguish these factor rings by the 
dimension of the Jacobson radical modulo its square; this dimension
is $3$ when $c=0$ and it is $2$ if $c=1$ (see \cite[4.5.2]{HZ}
for more details). Hence, given a block with dihedral defect group
and two simple modules, these results allow in principal to decide 
whether the scalar is $0$ or $q$, at least if one is able to 
explicitly compute the first K\"ulshammer ideal for the block
in question. 

\begin{Rem} Our method cannot treat the case where $q=p^m\equiv \pm 3\
\text{mod}\ 8$. The main reason is that K\"ulshammer ideals
are ideals of the center of the block in question and that when
$q=p^m\equiv \pm 3\ \text{mod}\ 8$, the defect groups are of order
$8$ in which case the center is too small for applying the results
of \cite{HZ}.\end{Rem}

This article is organized as follows. We give a short introduction
to K\"ulshammer ideals in Section \ref{Sec:Reynolds} and some basic
facts about the groups $PGL_2(q)$ are collected in 
Section \ref{Sec:groups}. 
Then Section \ref{Sec:mainproof} contains the proof of the theorem modulo
a key proposition whose proof is given in Section \ref{Sec:keyprop}. 
For simplicity, we concentrate on the case   $q\equiv 1\ \text{mod}\
8$ in these two sections.  The other case $q\equiv -1\ \text{mod}\ 8$ 
is similar and we will state the corresponding results without proofs   
in the final section.


\section{K\"ulshammer ideals (a.k.a. Generalized Reynolds ideals)} 
\label{Sec:Reynolds}

 Let $k$ be an algebraically closed field of positive
characteristic $p>0$. Let $A$ be  a (finite-dimensional) symmetric
 algebra, i.e. there exists a non-degenerate bilinear form
$(\ ,\ ): A\times A \rightarrow k$
such that for 
$a, b, c\in A$, $$(a, b)=(b,a)\ ~\text{and}\ ~  (ab,c)=(a, bc).$$
Denote by $K(A)$ the vector space generated by all commutators
$[a,b]=ab-ba$ with $a, b\in A$. For $n\geq 0$, we define
$$T_n(A)=\{x\in A\,|\, x^{p^n}\in K(A)\}.$$ 
The $n$-th K\"ulshammer ideal of $A$ is defined as
the orthogonal space (with respect to the symmetrizing 
form on $A$),
$$T_n^\bot(A)=\{x\in A\,|\, (x,  y)=0 \mbox{ for all } y\in T_n(A)\}.$$
We then have the following fundamental lemma (which is not difficult
to prove).

\begin{Lem}[\cite{K2}, no.(36)]  The subspaces $T_n^\bot(A)$ form a
decreasing sequence of ideals of the center $Z(A)$
$$Z(A)=K(A)^\bot=T_0^\bot(A) \supseteq T_1^\bot(A)\supseteq
T_2^\bot(A)\supseteq \cdots.$$  \end{Lem}

We will often consider the factor rings
$\overline{Z}(A):=Z(A)/T_1^\bot(A)$, and their Jacobson
radicals. For an algebra $B$ we denote by $J(B)$ the Jacobson 
radical. 
\smallskip

We illustrate the above notions using the typical examples of  group
algebras.
Let $G$ be a finite group and  $A=kG$ the group algebra. 
Then $A$ is a symmetric
algebra via the following paring:\\
    $$(\ ,\ ): kG\times kG \rightarrow k~~,~~~~
  (g, h)=\delta_{g,h^{-1}}=\left\{ \begin{array} {ll}
                  1   &        \mbox{if} \  \mbox{$g=h^{-1}$} \\
                  0 & \mbox{otherwise}
\end{array}\right.$$
for $g,h\in G$ and extension by linearity.
Let $X$ be a subset of $G$. We introduce the following notations:
 $$X^+=\sum_{x\in  X}x, \ \text{and}\
 X^{p^{-n}}=\{g\in G:g^{p^n}\in X\}.$$
 The K\"ulshammer ideals $T_n^\bot(kG)$
admit a nice description (\cite{K2}, no.(38)): for every
$n\ge 0$, the vector space $T_n^\bot(kG)$  has a 
basis $\{ (C^{p^{-n}})^+: C\in \mathcal{C}l(G) \}$
where $\mathcal{C}l(G)$ denotes the set of conjugacy classes of $G$.

\begin{Ex} \label{ex-cyclic}
Let $k$ be an
algebraically closed field $k$ of characteristic 2. 
We are going to compute K\"ulshammer ideals
for a cyclic group $G=C_{2^m}=\langle g\rangle$ 
of order $2^m$ with $m\geq 1$. In particular we
shall look at the factor ring
$\overline{Z}(kG):=Z(kG)/T_1^\bot(kG)$ and its
Jacobson radical.
 
The first K\"ulshammer ideal $T_1^{\perp}(kC_{2^m})$ has a vector
space basis of the form 
$$\{(C^{2^{-1}})^+: C\in \mathcal{C}l(G) \} =
\{ g^{j}+g^{2^{m-1}+j}\,|\,0\le j\le 2^{m-1}-1\}.$$
In fact, for a conjugacy class $C=\{g^i\}$ of $G$, 
the set $C^{2^{-1}} = \{y\in G\,|\,y^2=g^i\}$
is empty for $i$ odd, and consists of $g^{i/2}$ and $g^{2^{m-1}+i/2}$
if $i$ is even. In particular,
$$\dim T_1^{\perp}(kC_{2^m}) = 2^{m-1}~~~\mbox{and also}~~~
\dim \overline{Z}(kG)=2^{m-1}.$$
For the Jacobson radical we then get
$$\dim
J(\overline{Z}(kG))=2^{m-1}-1~~,~~~\dim
J^2(\overline{Z}(kG))=\max(0, 2^{m-1}-2),
$$  
and
$$\dim
J(\overline{Z}(kG))/J^2(\overline{Z}(kG))=
\left\{ \begin{array}{ll} 1 & \mbox{if $m>1$} \\
                          0 & \mbox{if $m=1$}
\end{array} \right.
.$$ 

We will use these calculations in the third section.
\end{Ex}

\begin{Rem}
K\"ulshammer ideals are known to have good multiplicative properties.
Let $A$ and $B$ be two symmetric $k$-algebras. Then
$A\times B$ is also symmetric via the obvious  bilinear form and
$$Z(A\times B)\cong Z(A)\times Z(B)$$
$$T_n^\bot(A\times B)\cong T_n^\bot(A)\times T_n^\bot(B)$$
$$\overline{Z}(A\times B)\cong \overline{Z}(A)\times \overline{Z}(B)$$
$$J(\overline{Z}(A\times B))\cong J(\overline{Z}(A))\times J(\overline{Z}(B))$$
$$J(\overline{Z}(A\times B))/J^2(\overline{Z}(A\times B))\cong 
J(\overline{Z}(A))/J^2(\overline{Z}(A))\times
J(\overline{Z}(B))/J^2(\overline{Z}(B))$$
We leave the proof of these easy facts to the reader.
\end{Rem}

We can now state the theorem of A.~Zimmermann cited in the
introduction saying that the K\"ulshammer ideals are derived
invariants. 

\begin{Thm} (\cite{Z1}) Let $k$ be an algebraically closed field of
characteristic $p>0$. Let $A$ and $B$ be symmetric $k$-algebras. If
$A$ and $B$ are derived equivalent (i.e., their derived module
categories 
$\mathcal{D}^b(A)\simeq \mathcal{D}^b(B)$ are equivalent as 
triangulated categories), then there exists an
isomorphism  $\varphi : Z(A) \rightarrow Z(B)$ such that
$\varphi(T_n^\bot(A))=T_n^\bot(B)$ for any $n\geq 0$.
\end{Thm}





From now on, $k$ denotes
an algebraically closed field of characteristic 2.

For proving the main result of this paper it will be crucial
to be able to decide, given a particular dihedral block,
which scalar occurs in the relation. This can be read off from
the factor rings modulo the first K\"ulshammer ideals,  
by the following recent result of the first named author and 
A.\,Zimmermann. Recall that all blocks with dihedral defect 
group of order $2^n$ occur among the algebras $D(2\mathcal{A})^s(c)$ 
defined in the introduction. Note that the following result
only applies for dihedral defect groups of order at least 16. 

 \begin{Thm} (\cite{HZ}) \label{HZ} Let $s={2^{n-2}}$ with
$n\geq 4$. Denote $A^s_c=D(2\mathcal{A})^s(c)$ and
$\overline{Z}_c=Z(A^s_c)/T_1^\bot(A^s_c)$. Then
$$\dim J(\overline{Z}_0)/J^2(\overline{Z}_0)=3, \ \text{and}\
\dim J(\overline{Z}_1)/J^2(\overline{Z}_1)=2$$
\end{Thm}
The main step of the proof of Theorem~\ref{main} is to calculate in
an undirect way the dimension of $
J(\overline{Z}(B_0))/J^2(\overline{Z}(B_0))$ for the principal block
$B_0$ of $kPGL_2(q)$.

\section{Some group-theoretic facts}
\label{Sec:groups}

In this section we collect some basic facts about projective
general linear groups. Most of them are well known, so we only
give some indications of proofs. 

Let $q=p^n$ be a prime power for $p$ an odd prime
number. The group $PGL_2(q)$ is defined as the factor group
of the general linear group $GL_2(q)$ over the finite field 
$\mathbb{F}_q$ modulo the center, i.e., the normal subgroup
of all scalar multiples of the identity matrix. In particular,
the group $PGL_2(q)$ has order $q(q+1)(q-1)$. 
  
Denote by $\sigma$ a generator of the multiplicative group
$\mathbb{F}_{q^2}^*$ of invertible elements of the finite field
$\mathbb{F}_{q^2}$. We set $\tau:=\sigma^{q+1}$; note that $\tau$
is a generator of $\mathbb{F}_{q}^*$. Moreover, we denote by $\epsilon$ 
(resp. $\eta$) a
$(q+1)$-th (resp. $(q-1)$-th) primitive root of $1$ in $\mathbb{C}$.

The first table gives the conjugacy  classes of $PGL_2(q)$ , where
$\lambda_1$ and $\lambda_2$ are two eigenvalues and where the last
column gives the order of the centralizer of a representative of a
conjugacy class.
\begin{table}[h]\caption{Conjugacy classes of  $PGL_2(q)$ with $q=p^n$
odd}
\label{table-conjclass}
\begin{tabular} {|c|c|c|}
\hline && \\
conjugacy  class $K$ &representative $x_K$ & $|C_G(x_K)|$ \\ &&\\
  \hline  \hline &&\\     $\lambda_1=\lambda_2 \in
  \mathbb{F}_q^*$ &&\\
  \text{semisimple}
& \raisebox{2.3ex}[0pt]{$A_1=\left[\begin{array}{cc}
1&0\\0&1\end{array}\right]$}&
    \raisebox{2.3ex}[0pt]{$q(q+1)(q-1)$}\\
 &&\\
        \hline &&\\
           $\lambda_1=\lambda_2 \in
  \mathbb{F}_q^*$ &&\\
  \text{nonsemisimple}
& \raisebox{2.3ex}[0pt]{$A_2=\left[\begin{array}{cc}
1&1\\0&1\end{array}\right]$}&
    \raisebox{2.3ex}[0pt]{$q$}\\
 &&\\
 \hline && \\
 $ \lambda_1\neq \pm\lambda_2  $  &$A_{3,i}=\left[\begin{array}{cc}
       1&0\\0&\tau^i\end{array}\right]$
 &  \\
 $ \lambda_1, \lambda_2  \in \mathbb{F}_q^*$  & $1\leq i\leq \frac{q-3}{2}$,
 & \raisebox{2.1ex}[0pt]{$q-1$}\\
 &&\\
 \hline && \\
 $    \lambda_1=-\lambda_2 \in \mathbb{F}_q^*$
 &
$A_{3, \frac{q-1}{2}}=\left[\begin{array}{cc}
 1&0\\0&-1\end{array}\right]$&  $2(q-1)$\\
  &&\\
\hline &&\\
$ \lambda_1\neq \pm\lambda_2$   &
  $A_{4, j}=\left[\begin{array}{cc} 0&
  -\sigma^{j(q+1)}\\1&\sigma^j+\sigma^{jq}\end{array}\right]$
   &   \\
   $ \lambda_1, \lambda_2 \in\mathbb{F}_{q^2}-\mathbb{F}_q$ &
    $1\leq j\leq
  \frac{q-1}{2}$,
   & \raisebox{2.1ex}[0pt]{$q+1$} \\
  &&\\
\hline &&\\
$ \lambda_1=-\lambda_2 \in\mathbb{F}_{q^2}-\mathbb{F}_q$ & $A_{4,
\frac{q+1}{2}}=\left[\begin{array}{cc}
  0&\tau\\1&0\end{array}\right]$& $2(q+1)$ \\
  &&\\
\hline
\end{tabular}\end{table}

The second table is the ordinary character table of of  $PGL_2(q)$,
where the characters in the last row come from cuspidal
representations. The character table of $GL_2(q)$ can be found in
\cite{PS} and we only need to find out all ordinary characters which
  factor through the subgroup formed by scalar matrices.
Note that our notations may differ slightly from the notations in
\cite{PS}. In total there are $q+2$ ordinary irreducible characters
for $PGL_2(q)$ (the same number as conjugacy classes, of course). 

\begin{table}[h]\caption{Character table of  $PGL_2(q)$ with $q=p^n$
odd}

\begin{tabular} {|c||c|c|c|c|c|c |}\hline & &&&&&\\
&&& $A_{3,i}$&& $A_{4,j}$&\\
    &\raisebox{2.3ex}[0pt]{$A_1$}   &\raisebox{2.3ex}[0pt]{$A_2$}&   $ 1\leq i\leq \frac{q-3}{2} $
  &  \raisebox{2.3ex}[0pt]{$A_{3, \frac{q-1}{2}}$} &  $ 1\leq j\leq \frac{q-1}{2} $  &
     \raisebox{2.3ex}[0pt]{$A_{4, \frac{q+1}{2}}$ } \\ \hline
\hline & &&&&&\\   $1_G$ & 1& 1& 1&1&1&1\\
 \hline & &&&&&\\    $\theta$ & $q$ & 0&1& 1& -1& -1\\
 \hline & &&&&&\\    $sgn$ & 1& 1& $(-1)^i$&  $(-1)^{\frac{q-1}{2}}$ & $(-1)^j$& $(-1)^{\frac{q+1}{2}}$ \\
  \hline & &&&&&\\  $\theta\otimes sgn$ & $q$& 0&$(-1)^i$&  $(-1)^{\frac{q-1}{2}}$ & $(-1)^{j-1}$& $(-1)^{\frac{q-1}{2}}$\\
   &&&&&&\\   \hline &&&&&&\\
    $\mu_s$&&&&&&\\
    $ 1\leq s\leq \frac{q-3}{2}$  & \raisebox{2.3ex}[0pt]{$q+1$}& \raisebox{2.3ex}[0pt]{1}&
    \raisebox{2.3ex}[0pt]{$\eta^{si}+\eta^{-si}$} &\raisebox{2.3ex}[0pt]{$2(-1)^s$}&
    \raisebox{2.3ex}[0pt]{0}&\raisebox{2.3ex}[0pt]{0}\\
   &&&&&&\\  \hline &&&&&&\\
     $\chi_k$&&&&&&\\
      $ 1\leq k \leq  \frac{q-1}{2}$  & \raisebox{2.3ex}[0pt]{$q-1$}& \raisebox{2.3ex}[0pt]{-1}&
      \raisebox{2.3ex}[0pt]{0} &
 \raisebox{2.3ex}[0pt]{0}&\raisebox{2.3ex}[0pt]{$-\epsilon^{jk}-\epsilon^{-jk}$}&
 \raisebox{2.3ex}[0pt]{$-2(-1)^k$}\\ \hline
\end{tabular}\end{table}


The $2$-Sylow subgroups of $PGL_2(q)$ are known to be
dihedral groups. As
$|PGL_2(q)|=q(q+1)(q-1)$, if $2$-Sylow subgroups are of order $2^n$
with $n\geq 2$ (i.e. the principal block is of defect $n$), then we
can write $q-1=2^{n-1}q'$ with $q'$ odd.
Table~\ref{regular} gives all $2$-regular conjugacy classes (i.e.
conjugacy classes consisting of elements whose order is not divisible
by $2$). 

\begin{table}[h]\caption{$2$-regular conjugacy classes of 
$PGL_2(q)$ with $q=p^n
\equiv 1\ \text{mod}\ 8$}\label{regular}

\begin{tabular} {|c|c|} \hline &\\
conjugacy class  $K$   &representative  $x_K$ \\
  \hline \hline &\\ 
$\lambda_1=\lambda_2  \in \mathbb{F}_q^*, \text{ semisimple} $
    & $A_1=\left[\begin{array}{cc} 1&0\\0&1\end{array}\right]$ \\

      \hline & \\
        $\lambda_1=\lambda_2 \in \mathbb{F}_q^*$, \text{nonsemisimple} 
&$A_2=\left[\begin{array}{cc} 1&1\\0&1\end{array}\right]$ \\
 \hline & \\    $ \lambda_1\neq \lambda_2 \in \mathbb{F}_q^*$  &
 $A_{3,2^{n-1}i'}=\left[\begin{array}{cc}
       1&0\\0&\tau^{2^{n-1}i'}\end{array}\right]$, $1\leq i'\leq
       \frac{q'-1}{2}$
 \\

\hline & \\  $ \lambda_1\neq \lambda_2
\in\mathbb{F}_{q^2}-\mathbb{F}_q$ &
  $A_{4, 2j'}=\left[\begin{array}{cc} 0&
  -\sigma^{2j'(q+1)}\\1&\sigma^{2j'}+\sigma^{2j'q}\end{array}\right]$, $1\leq j'\leq
  \frac{q-1}{4}$   \\ \hline

\end{tabular}\end{table}

One can then use ordinary characters to determine the block
structure of the group algebra 
$kPGL_2(q)$. We will have to distribute the ordinary irreducible
characters into $2$-blocks. This can be done using the following
well-known criterion (see for instance \cite[(3.19) THEOREM]{Nav}):  
{\em Two ordinary characters $\chi$ and $ \psi$ lie in the
same $2$-block if $$\sum_{x\in G^0}\chi(x)\psi(x^{-1})\neq 0$$ where
$G^0$ is the set of all $2$-regular elements.}

Furthermore, a block with dihedral defect group of order $2^n$ 
has precisely $2^{n-2}+3$ ordinary irreducible characters
(see for example, \cite[V.5.10 COROLLARY]{E1}). 

Table~\ref{block} gives the list of all $2$-blocks.  
Recall that the number $q'$ comes from the factorization
$q-1=2^{n-1}q'$. 
In this table
and in the sequel, $\sim $ means Morita equivalence. These
Morita equivalences can be deduced from a theorem of L.Puig \cite{P};
in fact the cyclic blocks occurring have only one simple module 
(i.e. one modular irreducible Brauer character), hence are 
nilpotent, and then by Puig's theorem they are Morita equivalent to the
group algebra of their defect groups.

\begin{table}[t]\caption{$2$-blocks for  $PGL_2(q)$ with 
$q=p^n \equiv  1\ \text{mod}\ 8$}\label{block}

\begin{tabular} {|c|c|c |}
\hline &&\\
 $2$-block & description & ordinary characters\\
  \hline \hline &&\\   principal block
  & dihedral block  &   $1_G, \theta, sgn, \theta\otimes sgn,$     \\
    $B_0 $
  &  of defect  $n$ &   $   \mu_{q't}, 1\leq t\leq 2^{n-2}-1,$  \\
      &&\\
    \hline   &&\\
          $B_{3,s} \sim kC_{2^{n-1}}$ &
          cyclic block &
    $ \mu_{q't-s}, 1\leq t\leq 2^{n-2} $  \\
 $ 1\leq s \leq \frac{q'-1}{2}$ &
            of defect $n-1$ &
    $   \mu_{q't+s}, 0\leq t\leq 2^{n-2}-1$  \\
     &&\\  \hline&&\\
       $B_{4,u}\sim kC_{2} $  & cyclic block    & \\
  $1\leq u\leq \frac{q-1}{4}$  &   of defect $1$   &\raisebox{2.3ex}[0pt]{$\chi_u,
  \chi_{\frac{q+1}{2}-u}$}\\ \hline
\end{tabular}

\end{table}

\section{Proof of the main theorem}
\label{Sec:mainproof}

In this section we shall closely look at the K\"ulshammer ideals
for the principal block of the group algebra of $PGL_2(q)$.
In particular, we shall prove our main result Theorem \ref{main}.
However, the proof of a key result, which is of a technical
nature, is postponed to the next section. 

By abuse of notation, we will freely use the notations for 
representatives of the conjugacy classes (like e.g. $A_{3,i}$)
now also for the entire conjugay class. Therefore, we denote
conjugay class sums by e.g. $A_{3,i}^+$, and also  
its image in the quotient
$\overline{Z}(kG)=Z(kG)/T_1^{\perp}(kG)$.

Recall K\"ulshammer's nice description (\cite{K2} no.(38))
of the ideals $T_1^{\perp}(kG)$
for group algebras $kG$. This has been mentioned already in 
Example \ref{ex-cyclic}; we restate it here for the special 
case we are needing: 
$T_1^\bot(kG)$ is the
vector space with basis $\{ (K^{2^{-1}})^+: K\in \mathcal{C}l(G) \}$
where $\mathcal{C}l(G)$ the set of conjugacy classes of $G$,
and where 
$$K^{2^{-1}} = \{ g\in G\,|\,g^2\in K\}.$$
Note that the sets $K^{2^{-1}}$ are closed under conjugation in $G$,
i.e. it suffices to consider the representatives of the conjugacy classes
from Table \ref{table-conjclass}.

Easy calculations then give
Table~\ref{ortho} below, 
describing the sets $K^{2^{-1}}$ for the groups $PGL_2(q)$. 

\begin{table}[ht]\caption{Computing $T_1^\bot (kG)$}\label{ortho}

\begin{tabular} {|c|c |} \hline &\\
 Conjugacy class  $K$ &$K^{2^{-1}}$\\
\hline \hline &\\$A_1$ &   $A_1$, $A_{3, \frac{q-1}{2}}$, $A_{4, \frac{q+1}{2}}$\\
&\\
 \hline &\\ $A_2$&  $A_2$\\
 &\\\hline &\\ $A_{3,i}$, $ 1\leq i\leq \frac{q-3}{2}, \text{even} $& $A_{3,\frac{i}{2}}$,
   $A_{3,\frac{q-1-i}{2}}$\\
 &\\\hline &\\$A_{3,i}$, $ 1\leq  i\leq \frac{q-3}{2}$,   \text{odd}   & $\emptyset$ \\
&\\\hline &\\  $A_{3, \frac{q-1}{2}}$ & $A_{3, \frac{q-1}{4}}$\\
 &\\ \hline &\\$A_{4,j}$, $ 1\leq j\leq \frac{q-1}{2}, \text{even} $  &
   $A_{4,\frac{j}{2}}$,
   $A_{4,\frac{q+1-j}{2}}$\\
  &\\\hline &\\ $A_{4,j}$, $ 1\leq j\leq \frac{q-1}{2}, \text{odd}  $  & $\emptyset$ \\
 &\\\hline &\\$A_{4, \frac{q+1}{2}}$ & $\emptyset$ \\ \hline

\end{tabular}\end{table}

From the results listed in Table~\ref{ortho}, and K\"ulshammer's
description of the ideals $T_1^{\perp}(kG)$ above, it
is easy to deduce the following result. Note that part
$(i)$ is just the well-known fact that the conjugacy class sums 
form a basis of the center of a group algebra. 

\begin{Lem} \label{lem-dim}
Let $G=PGL_2(q)$ where $q\equiv 1\mod 8$. 
\begin{itemize}
 \item[(i)] The center has $\dim Z(kG)=q+2$ and a basis of it is given by
$$\{A_1^+; A_2^+; A_{3,i}^+, 
1\leq i\leq \frac{q-3}{2};
  A_{3, \frac{q-1}{2}}^+;  A_{4,j}^+,  1\leq j\leq \frac{q-1}{2}; A_{4,
  \frac{q+1}{2}}^+\}.$$
  \item[(ii)] The first K\"ulshammer ideal has
$\dim T_1^\bot (kG)=\frac{q+3}{2} $ and a basis of it is given by
$$\{A_1^++A_{3, \frac{q-1}{2}}^++A_{4, \frac{q+1}{2}}^+;
  A_2^+; A_{3,i}^++A_{3,\frac{q-1}{2}-i}^+, 1\leq i\leq \frac{q-5}{4}; A_{3,
  \frac{q-1}{4}}^+;$$ $$
          A_{4,j}^++A_{4,\frac{q+1}{2}-j}^+, 1\leq j\leq
          \frac{q-1}{4}\}.$$
    \item[(iii)] The factor ring has 
$\dim Z(kG)/T_1^\bot (kG)=\frac{q+1}{2}$ and a basis of 
     it is  given by 
$$\{A_1^+;  A_{3,
    \frac{q-1}{2}}^+;
    A_{3,i}^+, 1\leq i\leq \frac{q-5}{4};  A_{4,j}^+, 1\leq j\leq
    \frac{q-1}{4}
    \}.$$
\end{itemize} $\hfill\Box$
\end{Lem}

Before proceeding to the proof of our main result
Theorem~\ref{main}, we consider a
small example.
\begin{Ex} Let $q=3^2=9$. Then $n=4$, $q'=1$, $\frac{q'-1}{2}=0$ and
 $$kG=kPGL_2(9)\sim B_0 \oplus (kC_2)^{\oplus 2}.$$
From the preceding lemma, we can read off the entries
of the following Table~\ref{PGL_2(9)}. For the last column on cyclic blocks
the entries have been computed in Example \ref{ex-cyclic}.

\begin{table}[h]\caption{$PGL_2(9)$} \label{PGL_2(9)} 
\begin{tabular}{|c||c|c|c| }

   \hline &&&\\

&$kG$&      $B_0 $ &     $(kC_2)^{2}$ \\

  \hline \hline &&&\\  \text{center} $Z$&  $11$ &  $? $   &$4$       \\

\hline &&&\\  $ T_1^\bot $& $6$ & ? &  $2$\\

\hline &&&\\ $\overline{Z}=Z/T_1^\bot $& $5$ & ? &  $2$\\ \hline

\end{tabular}\end{table}
One obtains that 
$\dim Z(B_0)/T_1^\bot(B_0)=5-2=3$ and then in particular
$$\dim J(\overline{Z(B_0)})/J^2(\overline{Z(B_0)})\leq 2.$$ 
This means that the principal $2$-block of $PGL_2(9)$ 
has scalar $c=1$, by Theorem~\ref{HZ}.\end{Ex}

For proving Theorem~\ref{main}, one needs to
compute the Jacobson radical and its square of the factor ring
$\overline{Z}:=Z(kG)/T_1^\bot (kG)$ for the group $G=PGL_2(q)$. 
The key step is the following proposition 
whose (technical) proof is
postponed to the next section (see Corollary~\ref{radical} and
Proposition~\ref{radicalsquare}).

\begin{Prop}[Key step] \label{key} With the above notations we have
$$\dim J(\overline{Z})= \frac{q-1}{4}-\frac{q'-1}{2}\,\,\,\,\,\,\,\,\,\,and
\,\,\,\,\,\,\,\,
\dim J^2(\overline{Z})= \frac{q-1}{4}-(q'+1).$$
\end{Prop}

\medskip
We now give the proof of theorem \ref{main} using the above key
proposition. 
Recall from Table \ref{block} 
the block decomposition, up to Morita equivalence:
$$kPGL_2(q) \sim B_0 \oplus (kC_{2^{n-1}})^{\oplus \frac{q'-1}{2}} 
\oplus (kC_{2})^{\oplus \frac{q-1}{4}}.$$ 

We collect the necessary information on the blocks and their
K\"ulshammer ideals in the following Table~\ref{complete};
the numbers in the table are dimensions. The entries for $kG$ 
are from the above key proposition
and Lemma \ref{lem-dim}, the entries in the last two columns
on cyclic blocks have been computed in Example \ref{ex-cyclic}. 

\begin{table}[ht]\caption{}\label{complete}
\begin{tabular}{|c||c|c|c|c |}

   \hline &$kG$&      $B_0 $ &   $(kC_{2^{n-1}})^{\oplus \frac{q'-1}{2}} $ & 
    $(kC_2)^{\frac{q-1}{4}}$ \\
\hline 
  \hline \text{center} $Z$&  $q+2$ & $2^{n-2}+3$ &$2^{n-1} \times \frac{q'-1}{2}$   
&$2\times \frac{q-1}{4}$       \\

\hline $\overline{Z}=Z/T_1^\bot $& $\frac{q+1}{2}$ & ? &  $2^{n-2}
\times \frac{q'-1}{2}$ &
$1\times \frac{q-1}{4}$\\

\hline $J(\overline{Z})$& $\frac{q-1}{4}-\frac{q'-1}{2}$ & ? &
$(2^{n-2}-1) \times \frac{q'-1}{2}$ &
$0$\\
\hline $J^2(\overline{Z})$& $\frac{q-1}{4}-(q'+1)$ & ? &
$(2^{n-2}-2) \times \frac{q'-1}{2}$ &
$0$\\
\hline $J(\overline{Z})/J^2(\overline{Z})$& $\frac{q'+3}{2}$ & ? & $
1\times \frac{q'-1}{2} $ & $0$\\ \hline

\end{tabular}\end{table}

From this table, one easily computes that  
$$\dim J(\overline{Z(B_0)})/J^2(\overline{Z(B_0)})=
\frac{q'+3}{2} - \frac{q'-1}{2} = 2.$$ 
Now using  
Theorem~\ref{HZ}, we can deduce that the scalar for the principal $2$-block
of $PGL_2(q)$ is indeed $c=1$, thus proving
the main result Theorem \ref{main}.

\section{Proof of the key proposition}
\label{Sec:keyprop}

Throughout this
section, we work in the quotient algebra 
$\overline{Z}(kG)=Z(kG)/T_1^{\perp}(kG)$. 
We will exhibit explicit basis for the radical $J(\overline{Z}(kG))$ and
the radical square $J^2(\overline{Z}(kG))$. For notational 
convenience, we allow the index $i$ in $A_{3, i}$ to 
be an arbitrary integer (note that by definition of the representative
$A_{3,i}$, the index can be taken modulo $q-1$, and moreover in 
$PGL_2(q)$ we have that $A_{3, i}=A_{3, q-1\pm i}$). 
The same convention applies to the index $j$ in $A_{4, j}$.

For the computations in this section, always keep in mind that
we are working in characteristic 2. 

\begin{Prop} \label{square} For all $i\in\mathbb{Z}$ the following 
equation holds in $\overline{Z}(kG)$:
$$(A_{3, i}^+)^2=\left\{ \begin{array} {r@{ ,  \quad }l}
                  A_{3, 2i}^+   &        \text{if }\    \frac{q-1}{4} \nmid i \\
                  0 &  \text{if}\    i=u \frac{q-1}{4}\   \text{with}\  u\  \text{odd}\\
                  A_{1}^+ &\text{if}\    i=u \frac{q-1}{4}\   \text{with}\  u \
                  \text{even}
\end{array}\right.$$

\end{Prop}

\begin{proof} For all $i\in\mathbb{Z}$ we have 
 \begin{eqnarray*} 
(A_{3, i}^+)^2&=&(\sum_{g\in G/C_G(A_{3,i})}g A_{3, i}g^{-1})^2
 =\sum_{g\in G/C_G(A_{3,i})}g A_{3, i}^2g^{-1}\\
 &=&\sum_{g\in
G/C_G(A_{3,i})}g A_{3, 2i}g^{-1}.\end{eqnarray*}
If $\frac{q-1}{4}\nmid i$, then it is easy to see that
$C_G(A_{3,i})=C_G(A_{3,2i})$, so we have
$$(A_{3, i}^+)^2=\sum_{g\in G/C_G(A_{3,i})}g A_{3, 2i}g^{-1}=
\sum_{g\in G/C_G(A_{3,2i})}g A_{3, 2i}g^{-1}=A_{3,2i}^+.$$
If $i=u \frac{q-1}{4}$  with $u$ odd, then
$|C_G(A_{3,2i})/C_G(A_{3,i})|=2$, and we have
$$(A_{3, i}^+)^2=2\sum_{g\in G/C_G(A_{3,2i})}g A_{3, 2i}g^{-1}=2A_{3,2i}^+=0.$$
If $i=u \frac{q-1}{4}$  with $u$ even, then
$$(A_{3, i}^+)^2=\sum_{g\in G/C_G(A_{3,i})}g A_{3, 2i}g^{-1}=
|G/C_G(A_{3,i})|A_1^+= \frac{q(q+1)}{2}A_1^+ =  A_1^+ ,$$
where for the last equality we use that $q\equiv 1 \mod 8$.
\end{proof}

As a consequence, the next result provides a basis of the radical
of the factor ring $\overline{Z}(kG)$, in terms of the basis
of $\overline{Z}(kG)$ given in Lemma \ref{lem-dim}(iii).
Before stating the result, let us introduce some more notations 
which will also be useful in the sequel.
For $0\leq s\leq \frac{q'-1}{2}$, denote
$$\mathcal{I}_s=\{ i: 1\leq i\leq \frac{q-5}{4}=2^{n-3}q'-1 ~\mbox{and}~
i\equiv \pm s \ \text{mod}\ q'\}.$$
 Let $\mathcal{I}_s^{\text{even}}$ (resp. $\mathcal{I}_s^{\text{odd}}$) 
be the set of even (resp. odd) numbers in
 $\mathcal{I}_s$. Note that $|\mathcal{I}_0|=2^{n-3}-1$ 
and $|\mathcal{I}_s|=2^{n-2}$
for $1\leq s\leq \frac{q'-1}{2}$.

\begin{Prop}\label{radical}   
  A basis of $ J(\overline{Z}(kG))$ is given by
  $$\{A_1^++A_{3,
    \frac{q-1}{2}}^+;
    A_{3,i}^+, i\in \mathcal{I}_0 ; A_{3,i}^++A_{3,s}^+,  i
    \in \mathcal{I}_s, i\neq s, 1\leq s\leq \frac{q'-1}{2}
    \}.$$
    As a consequence, $\dim J(\overline{Z}(kG))=
    \frac{q-1}{4}-\frac{q'-1}{2}$.
\end{Prop}

\begin{proof} We will first prove that these elements are nilpotent,
thus are contained in the radical. In fact, by Proposition \ref{square}
(and because we are working in characteristic 2) we have
$$(A_1^++A_{3, \frac{q-1}{2}}^+)^2=(A_1^+)^2+(A_{3,
\frac{q-1}{2}}^+)^2=A_1^++A_1^+=0.$$ 
Next, let $i\in \mathcal{I}_0$, i.e. $i=uq'$ with 
$1\leq u\leq 2^{n-3}-1$;
write $u=2^tu'$ with $u'$ odd. 
Then by Proposition \ref{square} we have 
$$( A_{3, uq'}^+)^{2^{n-2-t}}=( A_{3, 2^{n-3-t}q'2^tu'}^+)^{2}=
( A_{3, 2^{n-3}q'u'}^+)^{2}=( A_{3, u'\frac{q-1}{4}}^+)^2=0, $$  
where the first equality 
is obtained by iteration of Proposition~\ref{square} since 
$\frac{q-1}{4}=2^{n-3}q'$ doesn't divide 
$2^{n-4-t} q'2^tu'=2^{n-4}q'u'$.

Finally, let $i\in \mathcal{I}_s$ for $1\leq s\leq \frac{q'-1}{2}$,
and write $i=uq'\pm s$. Note that here $q'>1$ (otherwise such $s$ don't
exist).
Then, again by Proposition \ref{square}
\begin{eqnarray*}(A_{3,i}^++A_{3, s}^+)^{2^{n-1}}&=
     &(A_{3,i}^+)^{2^{n-1}}+(A_{3,s}^+)^{2^{n-1}}=
     A_{3,2^{n-1}q'u\pm 2^{n-1}s }^++ A_{3,2^{n-1}s}^+\\
     &=& A_{3,(q-1)u\pm 2^{n-1}s }^++ A_{3,2^{n-1}s}^+=
A_{3,2^{n-1}s}^++ A_{3,2^{n-1}s}^+\\
&=
     &0.\end{eqnarray*}
\marginpar{!!!}
For the second equality above note that we are indeed always in
the first case of Proposition \ref{square}, since
$2^{n-2}i\equiv \pm 2^{n-2}s \mod q'=\frac{q-1}{2^{n-1}}$,
and the representative satisfies $|\pm 2^{n-2}s|\le 2^{n-3}(q'-1)$
which can not become $0$ modulo $\frac{q-1}{4}=2^{n-3}q'$, 
since $q'>1$. For the fourth equality recall the definition of 
$A_{3,j}$: indices can be taken modulo $q-1$, and 
$A_{3,j}^+=A_{3,-j}^+$ since the matrices $A_{3,j}$ and
$A_{3,-j}$ are conjugate in $PGL_2(q)$. 

The elements listed in the statement of the proposition
are thus all in the radical and they are evidently
linearly independent. In total we have
     $$1+(2^{n-3}-1)+\frac{q'-1}{2}\times
     (2^{n-2}-1)=\frac{q-1}{4}-\frac{q'-1}{2}$$ elements. 
But on the other hand, 
from the block structure, the radical has at
     most the dimension
$$\frac{q+1}{2}-1-\frac{q'-1}{2}-\frac{q-1}{4}=\frac{q-1}{4}-\frac{q'-1}{2}$$ 
which is 
the dimension of $\overline{Z}(kG)$ minus the number of blocks
(cf. Table \ref{block}).
So the result follows.
\end{proof}

\begin{Rem}\label{autre} It is easy to see that for 
$1\leq s\leq \frac{q'-1}{2}$,  
there is only one number in
$\mathcal{I}_s$ which is divisible by $2^{n-2}$, denoted by
$\varphi(s)$ (not to be confused with Euler's totient function). 
We can replace the element $A_{3,i}^++A_{3,s}^+, i
    \in \mathcal{I}_s, i\neq s$ by
$A_{3,i}^++A_{3, \varphi(s) }^+, i
    \in \mathcal{I}_s, i\neq \varphi(s)$ in the basis. We will use this
point in the proof of Proposition~\ref{radicalsquare}. Furthermore,
if $i, j\in \mathcal{I}_s$ for $1\leq s\leq \frac{q'-1}{2}$, then
$$A_{3,i}^++A_{3, j  }^+=A_{3,i}^++A_{3, s}^++A_{3,j}^++A_{3, s}^+\in
J(\overline{Z}(kG)).$$
\end{Rem}

We now turn to studying 
the square of the radical. To this end we first need some preparations.

\begin{Lem}  \label{repre} For $1\leq i\leq \frac{q-3}{2}$ we have
$$A_{3, i}^+=\sum_{\alpha, \beta \in \mathbb{F}_q}
\left(
\begin{array} {r@{   \quad }l}
1+\alpha\beta-\alpha\beta \tau^i &      -\beta+\beta\tau^i \\
 \alpha(1+\alpha\beta)(1-\tau^i)  &      
-\alpha\beta+(1+\alpha\beta)\tau^i
\end{array} \right)
+ \sum_{  \gamma \in \mathbb{F}_q} \left( \begin{array} 
{r@{  \quad }l}
  \tau^i&     0\\
 \gamma(\tau^i-1) &      1
\end{array}\right)
$$
\end{Lem}

\begin{proof} As usual, let $G:=PGL_2(q)=GL_2(q)/\mathbb{F}_q^*$. 
Consider the following subgroups of $GL_2(q)$
$$T=\bigg\{\left( \begin{array} {r@{ ~  \quad }l}
                   a  &        0 \\
                  0 &  b
\end{array}\right), a, b \in \mathbb{F}_q^*\bigg\} \mbox{~~~~~and~~~~~}
B=\bigg\{\left( \begin{array} {r@{ ~  \quad }l}
                   a  &       c \\
                  0 &  b
\end{array}\right), a, b  \in \mathbb{F}_q^*, c\in \mathbb{F}_q\bigg\}.$$
Then for $1\leq i\leq \frac{q-3}{2}$ we have $C_G(A_{3,
i})=T/\mathbb{F}_q^*$, hence
$G/C_G(A_{3,i})\cong GL_2(q)/T.$ 
Representatives of $GL_2(q)/B$ can be chosen as
$$\bigg\{\left( \begin{array} {r@{ ~  \quad }l}
1&     0\\
 \alpha&     1
\end{array}\right), \alpha\in \mathbb{F}_q; \left( \begin{array} {r@{ ~  \quad }l}
0&     1\\
 1&      0
\end{array}\right)\bigg\}$$
and a set of representatives of $B/T$ can be chosen as
$$\bigg\{\left( \begin{array} {r@{ ~  \quad }l}
1&    \beta\\
 0&     1
\end{array}\right), \beta\in \mathbb{F}_q \bigg\}. $$ 
So a set of representatives of $G/C_G(A_{3,i})$ can be chosen as
$$\bigg\{\left( \begin{array} {r@{ ~  \quad }l}
1&     0\\
 \alpha&     1
\end{array}\right)\left( \begin{array} {r@{ ~  \quad }l}
1&    \beta\\
 0&     1
\end{array}\right),  \left( \begin{array} {r@{ ~  \quad }l}
0&    1\\
 1&   0
\end{array}\right) \left( \begin{array} {r@{ ~  \quad }l}
1&    \gamma\\
 0&     1
\end{array}\right), \alpha, \beta, \gamma\in \mathbb{F}_q \bigg\}. $$ 
The result then
follows easily from direct calculations.
\end{proof}

\begin{Prop} \label{squareII}
Let $i, j \in \mathbb{Z}$ such that 
$\frac{q-1}{4} \nmid
i, j, i\pm j$. Then in $\overline{Z}(kG)$ we have
$$   A_{3, i}^+ A_{3, j}^+=A_{3, i+j}^++A_{3, i-j}^+.$$
\end{Prop}

\begin{proof} First consider the product in the center $Z(kG)$; then 
$A_{3, i}^+ A_{3,
j}^+=\sum_{K}a_KK^+$, where $K$ runs through the set of conjugacy
classes. Then 
\begin{eqnarray*} A_{3, i}^+A_{3, j}^+&=&\sum_{g, h\in
G/C_G(A_{3,i})=G/C_G(A_{3,j})}g A_{3, i}g^{-1}h A_{3,
j}h^{-1}\\
&=&\sum_{g\in G/C_G(A_{3,i})}g A_{3, i}(\sum_{h\in G/C_G(A_{3,j})}
g^{-1}h A_{3, j}h^{-1}g)g^{-1} \\ &=&\sum_{g\in G/C_G(A_{3,i})}g
A_{3, i} A_{3, j}^+ g^{-1}.\end{eqnarray*} 
Therefore,
$$a_K=|G/C_G(A_{3,i})|\times |\text{elements of $K$ in } A_{3, i}A_{3,
j}^+|/|K|.$$ 

One can use this simple counting principle to
compute $a_K$. We will treat the most difficult case where $K=A_{3,
u}$ for $1\leq u\leq \frac{q-3}{2}$ and leave the 
other cases to the reader. 
The above formula for the coefficient $a_K$ now reads for $K=A_{3,u}$
as
$$a_K =  q(q+1)\times |\text{elements of $K$ in } A_{3, i}A_{3,
j}^+|/q(q+1) = |\text{elements of $K$ in } A_{3, i}A_{3,
j}^+|.$$ 

Now by Lemma~\ref{repre}
\begin{eqnarray*}A_{3, i}  A_{3, j}^+&=&
\sum_{\alpha, \beta\in \mathbb{F}_q} 
\left( \begin{array} {r@{ \quad }l}
1 &   0  \\0&        \tau^i
\end{array}\right)
\left( \begin{array} {r@{   \quad }l}
1+\alpha\beta-\alpha\beta \tau^j&      -\beta+\beta \tau^j\\
 \alpha(1+\alpha\beta)(1- \tau^j)&      -\alpha\beta+(1+\alpha\beta)
 \tau^j
\end{array}\right)\\ 
&& + \sum_{\gamma\in \mathbb{F}_q}  
\left( \begin{array} {r@{   \quad }l} 
1 &   0  \\ 0&        \tau^i
\end{array}\right)
\left( \begin{array} {r@{   \quad }l}
  \tau^j&     0\\
 -\gamma+ \gamma \tau^j&      1
\end{array}\right)\\
&=&\sum_{\alpha, \beta\in \mathbb{F}_q}  
\left( \begin{array} {r@{ \quad }l}
1+\alpha\beta-\alpha\beta \tau^j&      -\beta+\beta \tau^j\\
 \alpha(1+\alpha\beta)\tau^i(1-\tau^j)&      -\alpha\beta\tau^i+(1+\alpha\beta)
 \tau^{i+j}
\end{array}\right) \\& &+ \sum_{\gamma\in \mathbb{F}_q}
\left( \begin{array} {r@{   \quad }l}
  \tau^j&     0\\
 \tau^i\gamma(\tau^j-1)&       \tau^i
\end{array}\right).\end{eqnarray*}

Denote by $B$ the first matrix in the preceding formula and by $C$
the second matrix. If $B$ represents the same coset as 
$A_{3, u}$ in $PGL_2(q)$,
then there exists $\lambda \in \mathbb{F}_q^*$ which satisfies, by
considering the determinant and the trace
$$\lambda^2\tau^u=\tau^{i+j} ~~~\mbox{~~and~~}~~~  
\lambda
(1+\tau^u)=1+\tau^{i+j}+\alpha\beta (1-\tau^i)(1-\tau^j).$$ The case
that $i+j-u$ is odd is impossible, as $\lambda\in \mathbb{F}_q^*$ 
and $\tau$ is a generator of $\mathbb{F}_q^*$, i.e. the squares 
in $\mathbb{F}_q^*$ are given by the even powers of $\tau$.
So consider now the case where $i+j-u$ is even, then $\lambda=\pm
\tau^{\frac{i+j-u}{2}}$ and
$$ \pm (1+\tau^u)\tau^{\frac{i+j-u}{2}} =1+\tau^{i+j}+\alpha\beta
(1-\tau^i)(1-\tau^j).$$ 
If $ \pm (1+\tau^u)\tau^{\frac{i+j-u}{2}}
=1+\tau^{i+j}$, i.e. $u=\pm (i+j)\ \text{mod}\ q-1$, then 
$\alpha\beta=0$. We have $2q-1$ possibilities for the pair $(\alpha,
\beta)$, namely $(0,0);(0, \beta), \beta \in \mathbb{F}_q^*;
 (\alpha, 0),
 \alpha \in \mathbb{F}_q^*$.
If $ \pm (1+\tau^u)\tau^{\frac{i+j-u}{2}} \neq
1+\tau^{i+j}$,  then $$\alpha\beta=\frac{\pm
(1+\tau^u)\tau^{\frac{i+j-u}{2}} -
1-\tau^{i+j}}{(1-\tau^i)(1-\tau^j)}\neq 0$$ and we have $q-1$
possibilities for the pair $(\alpha, \beta)$. 

It is not difficult to see
that $C$ is similar to $A_{3, i-j}$ in $PGL_2(q)$ and so there are
$q$ possibilities for $\gamma$. The counting principle gives $a_K=1$
for $K=A_{3, i\pm j}^+$ and $a_K=0$ for $K=A_{3, u}$ with 
$\pm u\neq   i\pm j \ \text{mod}\ q-1$.
\end{proof}

The following crucial result determines a basis of the
square of the radical, in terms of the basis of the radical
obtained in Proposition \ref{radical}.

\begin{Prop}\label{radicalsquare}
  A basis of $ J^2(\overline{Z}(kG))$ is given by the union
  $\mathcal{B}_0 \bigcup  (\bigcup_{s=1}^{\frac{q'-1}{2}}
  \mathcal{B}_s)$ where
    $$\mathcal{B}_0=\{ A_{3, i}^+, i \in \mathcal{I}_0^{\text{even}};
    A_{3,q'}^++A_{3,i}^+, i \in \mathcal{I}_0^{\text{odd}}-\{q'\} \}$$
and for $1\leq s\leq \frac{q'-1}{2}$, if $s$ is odd,
$$\mathcal{B}_s=
\{A_{3, s}^++A_{3, i}^+, i \in \mathcal{I}_s^{odd}-\{ s\}, A_{3,
q'+s}^++A_{3, i}^+, i \in \mathcal{I}_s^{even}-\{ q'+s\}\}$$ and if
$s$ is even, $$\mathcal{B}_s= \{A_{3, s}^++A_{3, i}^+, i \in
\mathcal{I}_s^{even}-\{ s\}; A_{3, q'+s}^++A_{3, i}^+, i \in
\mathcal{I}_s^{odd} -\{ q'+s\}\}.$$
 As a consequence, $\dim
J^2(\overline{Z}(kG))= \frac{q-1}{4}-(q'+1)$

\end{Prop}

\begin{proof}
We will write $\overline{Z}=\overline{Z}(kG)$ in the following. 
We will first prove that the elements listed above are indeed in the 
square of the radical.

\textit{Case $\mathcal{B}_0$.} Let $i\in\mathcal{I}^{even}_0$,
i.e. $i=uq'$ for some $1\le u\le 2^{n-3}-1$ with $u$ even.  
Then by Proposition \ref{square}
$$A_{3,i}^+ = A_{3, uq'}^+=(A_{3, \frac{u}{2}q'}^+)^2 \in
J^2(\overline{Z}).$$ 
Now let $i\in\mathcal{I}_0\setminus \{q'\}$, i.e.
$i=uq'$ for some $1< u\leq 2^{n-3}-1$ where $u$ is odd.
Then by Proposition \ref{squareII} we have
$$ A_{3,q'}^++A_{3, i}^+=
A_{3,q'}^++A_{3, uq'}^+=A_{3, \frac{u+1}{2}q'}^+A_{3,
\frac{u-1}{2}q'}^+\in J^2(\overline{Z}).$$

\smallskip

\textit{Case $\mathcal{B}_s$ with $s$ odd.} 
For $i \in \mathcal{I}_s^{even}-\{q'+ s\}$ we claim that
$$A_{3, q'+s}^++A_{3, i}^+=(A_{3, \frac{q'+s}{2}}^++A_{3,
\frac{i}{2}}^+)^2\in J^2(\overline{Z})$$
(where the first equality holds by Proposition \ref{square}).  
In fact, write
$i=(2u+1)q'\pm s$ with $u\in \mathbb{Z}$. 
Then, if $i=(2u+1)q'+ s$, we get 
$$ A_{3, \frac{q'+s}{2}}^++A_{3,
\frac{i}{2}}^+= A_{3, \frac{q'+s}{2}}^++A_{3,
uq'+\frac{q'+s}{2}}^+\in J(\overline{Z})$$ 
by Remark~\ref{autre}, 
and we have 
$$ A_{3,  q'+s }^++A_{3,
i}^+= (A_{3, \frac{q'+s}{2}}^++A_{3,
uq'+\frac{q'+s}{2}}^+)^2\in J^2(\overline{Z});$$
similarly, if $i=(2u+1)q'-s$,
$$ A_{3, \frac{q'+s}{2}}^++A_{3,
\frac{i}{2}}^+=A_{3, -\frac{q'+s}{2}}^++A_{3,
(u+1)q'-\frac{q'+s}{2}}^+ \in J(\overline{Z})$$ and we have 
$$ A_{3,  q'+s}^++A_{3,i}^+=(A_{3, -\frac{q'+s}{2}}^++A_{3,
(u+1)q'-\frac{q'+s}{2}}^+)^2 \in J^2(\overline{Z}).$$
Suppose now that $i \in
\mathcal{I}_s^{odd}-\{s\}$. Then write $i=2uq'\pm s$ for $u\in
\mathbb{Z}$. Obviously $\frac{q-1}{4}$ doesn't divide $u$. If $u$ is even, 
then by Proposition \ref{squareII} and using the notation from 
Remark \ref{autre} we get 
$$J^2(\overline{Z}) \ni
A_{3,uq'}^+(A_{3, uq'\pm s}^++A_{3, \varphi(s)}^+)= A_{3, s}^++A_{3,
2uq'\pm s}^++A_{3, uq'+\varphi(s)}^++A_{3, uq'-\varphi(s)}^+.$$ As
$uq'\pm \varphi(s)$ are even,  $$A_{3, uq'+\varphi(s)}^++A_{3,
uq'-\varphi(s)}^+=(A_{3, \frac{u}{2}q'+\frac{\varphi(s)}{2}}^++A_{3,
\frac{u}{2}q'-\frac{\varphi(s)}{2}}^+)^2\in J^2(\overline{Z})$$ and
therefore $A_{3, s}^++A_{3, 2uq'+s}^+\in J^2(\overline{Z})$. 

If $u$ is odd, then $$J^2(\overline{Z}) \ni
A_{3,uq'}^+(A_{3, uq'\pm s}^++A_{3, s}^+)= A_{3, s}^++A_{3,
2uq'\pm s}^++A_{3, uq'+s}^++A_{3, uq'-s}^+.$$ As
$uq'\pm s$ are even,  $$A_{3, uq'+\varphi(s)}^++A_{3,
uq'-\varphi(s)}^+=(A_{3, \frac{u-1}{2}q'+\frac{q'+s}{2}}^++A_{3,
\frac{u+1}{2}q'-\frac{q'+s}{2}}^+)^2\in J^2(\overline{Z})$$ and
therefore $A_{3, s}^++A_{3, 2uq'+s}^+\in J^2(\overline{Z})$.

\smallskip

 \textit{Case $\mathcal{B}_s$ with $s$ even.}  For $i \in
\mathcal{I}_s^{even}-\{ s\}$, we have $$A_{3,  s}^++A_{3,
i}^+=(A_{3, \frac{ s}{2}}^++A_{3, \frac{i}{2}}^+)^2\in
J^2(\overline{Z}),$$ since if we write $i=2uq'\pm s$, then
$\frac{i}{2}=uq'\pm \frac{s}{2}\in \mathcal{I}_{\frac{s}{2}}$.

Suppose now that   $i \in \mathcal{I}_s^{odd}-\{q'+s\}$. Then write
$i=uq'\pm s$ for $u\in \mathbb{Z}$ with $u\neq 1$ odd. Then either $4\mid
u-1$ or  $4\mid u+1$. Suppose $4\mid u-1$, the other case being
similar. For $i=uq'+s$ we have (since $u\neq 1$)
$$J^2(\overline{Z}) \ni A_{3,\frac{u-1}{2}q'}^+(A_{3,
  \frac{u+1}{2}q'+ s}^++A_{3, \varphi(s)}^+)$$ $$= A_{3, uq'+
s}^++A_{3, q'+s}^++A_{3, \frac{u-1}{2}q'+\varphi(s)}^++A_{3,
\frac{u-1}{2}q'-\varphi(s)}^+.$$ As $\frac{u-1}{2}q'\pm \varphi(s)$
are even, $A_{3, \frac{u-1}{2}q'+\varphi(s)}^++A_{3,
\frac{u-1}{2}q'-\varphi(s)}^+\in J^2(\overline{Z})$ and therefore
$A_{3,i}^+ + A_{3,q'+s} = A_{3,uq'+s}^++A_{3, q'+s}^+\in J^2(\overline{Z})$. 

For the case $i=uq'-s$, we prove firstly that 
$A_{3, q'+s}^++A_{3,
q'-s}^+\in J^2(\overline{Z})$. In fact,  $$J^2(\overline{Z}) 
\ni A_{3, q'}^+(A_{3,s}^++A_{3, q'+s}^+) = A_{3,  q'+
s}^++A_{3, q'-s}^++A_{3, 2q'+s}^++A_{3, s}^+$$ 
and hence
$$A_{3, 2q'+s}^++A_{3, s}^+=(A_{3, q'+\frac{s}{2}}^++A_{3,
\frac{s}{2}}^+)^2\in J^2(\overline{Z}).$$ 
Now we consider $$J^2(\overline{Z}) \ni
A_{3,\frac{u-1}{2}q'}^+(A_{3,
  -\frac{u+1}{2}q'+ s}^++A_{3, \varphi(s)}^+)$$ $$= A_{3, uq'-
s}^++A_{3, -q'+s}^++A_{3, \frac{u-1}{2}q'+\varphi(s)}^++A_{3,
\frac{u-1}{2}q'-\varphi(s)}^+$$
$$= A_{3, uq'-
s}^++A_{3, q'+s}^++A_{3, q'+s}^++A_{3, -q'+s}^++A_{3, \frac{u-1}{2}q'+\varphi(s)}^++A_{3,
\frac{u-1}{2}q'-\varphi(s)}^+.$$ 
As above, $A_{3,
\frac{u-1}{2}q'+\varphi(s)}^++A_{3, \frac{u-1}{2}q'-\varphi(s)}^+\in
J^2(\overline{Z})$, and then we also get   
$A_{3, uq'-s}^++A_{3, q'+s}^+\in J^2(\overline{Z})$.
\smallskip

Recall that $|\mathcal{I}_0|=2^{n-3}-1$ and $\mathcal{I}_s|=2^{n-2}$
for $s>1$. Then the total number of elements in the square 
of the radical listed in the proposition is
$$((2^{n-3}-1)-1)+\frac{q'-1}{2}\times (2^{n-2}-2)=\frac{q-1}{4}-(q'+1)$$
and moreover, these elements are evidently linearly independent. 
So
$$\dim J^2(\overline{Z}(kG))\geq \frac{q-1}{4}-(q'+1)$$ and
then we deduce from Proposition \ref{radical} that
$$\dim J(\overline{Z}(kG))/J^2(\overline{Z}(kG))\leq
 (\frac{q-1}{4}-\frac{q'-1}{2})-(\frac{q-1}{4}-(q'+1))=\frac{q'+3}{2}.$$
We thus have by the already proven information in
Table~\ref{complete}  
$$\dim
 J(\overline{Z}(B_0))/J^2(\overline{Z}(B_0))\leq
 \frac{q'+3}{2}-1\times \frac{q'-1}{2}=2.$$ 
However, by Theorem~\ref{HZ}, this dimension can only be $2$ or $3$,
thus
 $$\dim
 J(\overline{Z}(B_0))/J^2(\overline{Z}(B_0))=2,$$
 which implies 
 $$\dim J(\overline{Z}(kG))/J^2(\overline{Z}(kG))=\frac{q'+3}{2}$$
 and then also completes the proof of the key Proposition \ref{key}.
\end{proof}

\section{The  case $q\equiv -1\ \text{mod}\ 8$  }

In this section, we state without proofs the corresponding results for the case $q\equiv -1\ \text{mod}\ 8$.
Now $q+1=2^{n-1}q'$ with $q'$ odd where $n$ is the defect of the principal block.
Table~\ref{conjugacyclassesautre} gives the list of $2$-regular conjugacy classes in this case and Table~\ref{blockstructureautre} describes the block structure.

\begin{table}[h]\caption{$2$-regular conjugacy classes of 
$PGL_2(q)$ with $q=p^n
\equiv -1\ \text{mod}\ 8$} \label{conjugacyclassesautre}

\begin{tabular} {|c|c|} \hline &\\
conjugacy class  $K$   &representative  $x_K$ \\
  \hline \hline &\\      
$\lambda_1=\lambda_2  \in \mathbb{F}_q^*, \text{ semisimple} $
    & $A_1=\left[\begin{array}{cc} 1&0\\0&1\end{array}\right]$ \\

      \hline & \\
        $\lambda_1=\lambda_2 \in \mathbb{F}_q^*$, \text{nonsemisimple} 
&$A_2=\left[\begin{array}{cc} 1&1\\0&1\end{array}\right]$ \\
 \hline & \\    $ \lambda_1\neq \lambda_2 \in \mathbb{F}_q^*$  &
 $A_{3,2i'}=\left[\begin{array}{cc}
       1&0\\0&\tau^{2i'}\end{array}\right]$, $1\leq i'\leq
       \frac{q-3}{4}$
 \\

\hline & \\  $ \lambda_1\neq \lambda_2
\in\mathbb{F}_{q^2}-\mathbb{F}_q$ &
  $A_{4, 2^{n-1}j'}=\left[\begin{array}{cc} 0&
  -\sigma^{2^{n-1}j'(q+1)}\\1&\sigma^{2^{n-1}j'}+\sigma^{2^{n-1}j'q}\end{array}\right]$, $1\leq j'\leq
  \frac{q'-1}{2}$   \\ \hline

\end{tabular}\end{table}

\begin{table}[t]\caption{$2$-blocks for  $PGL_2(q)$ with $q=p^n \equiv  -1\ \text{mod}\
8$} \label{blockstructureautre}

\begin{tabular} {|c|c|c |}
\hline &&\\
 $2$-block & description & ordinary characters\\
  \hline \hline &&\\   principal block
  & dihedral block  &   $1_G, \theta, sgn, \theta\otimes sgn,$     \\
    $B_0 $
  &  of defect  $n$ &   $   \chi_{q't}, 1\leq t\leq 2^{n-2}-1,$  \\
      &&\\
     \hline&&\\
       $B_{3,s}\sim kC_{2} $  & cyclic block    & \\
  $1\leq s\leq \frac{q-3}{4}$  &   of defect $1$   &\raisebox{2.3ex}[0pt]{$\mu_s,
  \mu_{\frac{q-1}{2}-s}$}\\  
\hline   &&\\
          $B_{4,u} \sim kC_{2^{n-1}}$ &
          cyclic block &
    $ \chi_{q't-u}, 1\leq t\leq 2^{n-2} $  \\
 $ 1\leq u \leq \frac{q'-1}{2}$ &
            of defect $n-1$ &
    $   \chi_{q't+u}, 0\leq t\leq 2^{n-2}-1$  
\\
     
     \hline

\end{tabular}

\end{table}

\begin{Prop}  \label{keyautre}
Let $G=PGL_2(q)$ where $q=p^n\equiv -1\mod 8$. 
\begin{itemize}
 \item[(i)] $\dim Z(kG)=q+2$ and a basis of it is $$\{A_1^+; A_2^+; A_{3,i}^+, 
1\leq i\leq \frac{q-3}{2};
  A_{3, \frac{q-1}{2}}^+;  A_{4,j}^+,  1\leq j\leq \frac{q-1}{2}; A_{4,
  \frac{q+1}{2}}^+\}.$$
  \item[(ii)] $\dim T_1^\bot (kG)=\frac{q+3}{2} $ and a basis of it is
$$\{A_1^++A_{3, \frac{q-1}{2}}^++A_{4, \frac{q+1}{2}}^+;
  A_2^+; A_{3,i}^++A_{3,\frac{q-1}{2}-i}^+, 1\leq i\leq \frac{q-3}{4};$$ $$
          A_{4,j}^++A_{4,\frac{q+1}{2}-j}^+, 1\leq j\leq
          \frac{q-3}{4}; A_{4,\frac{q+1}{4}}^+\}.$$
    \item[(iii)] $\dim Z(kG)/T_1^\bot (kG)=\frac{q+1}{2}$ and a basis of 
     it is  $$\{A_1^+;  A_{4,
    \frac{q+1}{2}}^+;
    A_{3,i}^+, 1\leq i\leq \frac{q-3}{4};  A_{4,j}^+, 1\leq j\leq
    \frac{q-3}{4}
    \}.$$
\item[(iv)] $\dim J(\overline{Z}(kG)) =\frac{q+1}{4}-\frac{q'-1}{2}$ and a  basis of it is given by
  $$\{A_1^++A_{4,
    \frac{q+1}{2}}^+;
    A_{4,j}^+, j\in \mathcal{J}_0 ; A_{4,j}^++A_{4,s}^+,  j
    \in \mathcal{J}_s-\{s\}, 1\leq s\leq \frac{q'-1}{2}
    \}$$ where for $0\leq s\leq \frac{q'-1}{2}$,  
 $$\mathcal{J}_s=\{ j: 1\leq j\leq \frac{q-3}{4}=2^{n-3}q'-1, 
j\equiv \pm s \ \text{mod}\ q'\}.$$

\item[(v)] $\dim J^2(\overline{Z}(kG))=\frac{q+1}{4}-(q'+1)$ and a basis of 
     it is given by  the union
  $\mathcal{C}_0 \bigcup  (\bigcup_{s=1}^{\frac{q'-1}{2}}
  \mathcal{C}_s)$ where
    $$\mathcal{C}_0=\{ A_{4, j}^+, j \in \mathcal{J}_0^{\text{even}};
    A_{4,q'}^++A_{4,j}^+, j \in \mathcal{J}_0^{\text{odd}}-\{q'\} \}$$
and for $1\leq s\leq \frac{q'-1}{2}$, if $s$ is odd,
$$\mathcal{C}_s=
\{A_{4, s}^++A_{4, j}^+, j \in \mathcal{J}_s^{odd}-\{ s\}, A_{4,
q'+s}^++A_{4, j}^+, j \in \mathcal{J}_s^{even}-\{ q'+s\}\}$$ and if
$s$ is even, $$\mathcal{C}_s= \{A_{4, s}^++A_{4, j}^+, j \in
\mathcal{J}_s^{even}-\{ s\}; A_{4, q'+s}^++A_{4, j}^+, j \in
\mathcal{J}_s^{odd} -\{ q'+s\}\}.$$
\end{itemize}

\end{Prop}

To prove (iv), (v) of the preceding proposition, one needs the following 
computational result.
\begin{Prop}
\begin{itemize} The following holds in  $\overline{Z}(kG)$:

\item[(i)]
For all $j\in\mathbb{Z}$  
$$(A_{4, j}^+)^2=\left\{ \begin{array} {r@{ ,  \quad }l}
                  A_{4, 2j}^+   &        \text{if }\    \frac{q+1}{4} \nmid j \\
                  0 &  \text{if}\    j=u \frac{q+1}{4}\   \text{with}\  u\  \text{odd}\\
                  A_{1}^+ &\text{if}\    j=u \frac{q+1}{4}\   \text{with}\  u \
                  \text{even}
\end{array}\right.$$
\item[(ii)] For $i, j \in \mathbb{Z}$,  if $\frac{q+1}{4} \nmid
i, j, i\pm j$, then in $\overline{Z}(kG)$
$$   A_{4, i}^+\times A_{4, j}^+=A_{4, i+j}^++A_{4, i-j}^+$$

\end{itemize}

\end{Prop}

The proof of the main theorem~\ref{main} in this case then follows from Table~\ref{completeautre} whose entries can be deduced 
from Proposition ~\ref{keyautre}.
\begin{table}[ht]\caption{} \label{completeautre}
\begin{tabular}{|c||c|c|c|c |}

   \hline &$kG$&      $B_0 $ &  
    $(kC_2)^{\frac{q-3}{4}}$ & $(kC_{2^{n-1}})^{\oplus \frac{q'-1}{2}} $  \\
\hline 
  \hline \text{center} $Z$&  $q+2$ & $2^{n-2}+3$    
&$2\times \frac{q-3}{4}$ &$2^{n-1} \times \frac{q'-1}{2}$      \\

\hline $\overline{Z}=Z/T_1^\bot $& $\frac{q+1}{2}$ & ? &
$1\times \frac{q-3}{4}$ &  $2^{n-2}
\times \frac{q'-1}{2}$ \\

\hline $J(\overline{Z})$& $\frac{q+1}{4}-\frac{q'-1}{2}$ & ? &
$0$ &
$(2^{n-2}-1) \times \frac{q'-1}{2}$ \\
\hline $J^2(\overline{Z})$& $\frac{q+1}{4}-(q'+1)$ & ? &
$0$&
$(2^{n-2}-2) \times \frac{q'-1}{2}$ \\
\hline $J(\overline{Z})/J^2(\overline{Z})$& $\frac{q'+3}{2}$ & ?  & $0$& $
1\times \frac{q'-1}{2} $\\ \hline

\end{tabular}\end{table}


\begin{thebibliography}{80}

\bibitem{BHZ}C.~Bessenrodt, T.~Holm and A.~Zimmermann,  
\emph{Generalized Reynolds ideals
for non-symmetric algebras}, J.~Algebra  \textbf{312} (2007) 985-994

\bibitem{BHHKM} T.~Breuer, L.~H\'{e}thelyi, E.~Horv\'{a}th. B.~K\"{u}lshammer and 
J.~Murray,  
\emph{Cartan invariants and central ideals of group algebras},
J.~Algebra (3) \textbf{296} (2006), 177-195

\bibitem{Broue} M.~Brou\'e,   \emph{Isom\'etries parfaites, types de blocs, 
cat\'egories d\'eriv\'ees},
Ast\'erisque  \textbf{(181-182}) (1990), 61-92

\bibitem{D66}
  E.~C.~Dade,   \emph{Blocks with cyclic defect groups},  
Ann.~Math.~\textbf{84}   (1966),  20-48

\bibitem{E79} K.~Erdmann, \emph{Blocks whose defect groups are 
Klein four groups}, J. Algebra \textbf{59}
(1979), no. 2, 452-465

\bibitem{E82} K.~Erdmann, \emph{Blocks whose defect groups are 
Klein four groups: a correction}, J. Algebra \textbf{76} (1982), 
no. 2, 505-518

\bibitem{E87} K.~Erdmann, \emph{Algebras and dihedral defect groups}, 
Proc. London Math. Soc. (3) \textbf{54} (1987), no. 1, 88-114

\bibitem{E88} K.~Erdmann, \emph{Algebras and semidihedral 
defect groups. I}, Proc. London Math. Soc. (3) \textbf{57} 
(1988), no. 1, 109-150

\bibitem{E90} K.~Erdmann, \emph{Algebras and semidihedral 
defect groups. II}, Proc. London Math. Soc. (3) \textbf{60} 
(1990), no. 1, 123-165

\bibitem{E882} K.~Erdmann, \emph{Algebras and quaternion defect 
groups. I, II},  Math. Ann. \textbf{281} (1988), no. 4, 
545-560, 561-582

\bibitem{E1} K.~Erdmann, \emph{Blocks of Tame Representation Type 
and Related Algebras}, Lecture Notes in Mathematics 1428, 
Springer-Verlag Berlin 1990

\bibitem{E2} K.~Erdmann, \emph{On $2$-modular representations of $GU_2(q)$,
$q=3\ \text{mod}\  4$   }, Comm.~Algebra \textbf{20}(12) (1992),
3479-3502

\bibitem{H1}D.~Happel, \emph{On the derived category of a 
finite-dimensional algebra},
Comment.~Math.~Helv.~\textbf{62} (1987), 339-389

\bibitem{H2} D.~Happel, \emph{Triangulated Categories in the 
Representation Theory of Finite Dimensional
Algebras}, LMS Lecture Note Series 119,   Cambridge University
Pr\`{e}ss 1988

\bibitem{HHKM} L.~H\'{e}thelyi, E.~Horv\'{a}th. B.~K\"{u}lshammer 
and J.~Murray,
\emph{Central ideals and Cartan invariants of symmetric algebras},
J.~Algebra~\textbf{293 }(2005), 243-260

\bibitem{Holm1} T.~Holm, \emph{Derived equivalent tame blocks},
J.~Algebra~ \textbf{194} (1997) no.1, 178-200

\bibitem{Holm2} T.~Holm, \emph{Derived equivalence classification 
of algebras of dihedral, semidihedral and quaternion type},
J.~Algebra~\textbf{211} (1999) no.1, 159-205

\bibitem{HS} T.~Holm and A.~Skowro\'{n}ski, \emph{Derived equivalence
classification of symmetric algebras of domestic type},  
J.~Math.~Soc.~Japan \textbf{58} (2006), no.4, 1133-1149

 \bibitem{HZ} T.~Holm and A.~Zimmermann, 
\emph{Generalized Reynolds ideals and derived equivalences for 
algebras of dihedral and semidihedral type}, J.~Algebra, to appear

\bibitem{Kauer}
M. Kauer, Derived equivalence of graph algebras.
{\em Trends in the representation theory of finite-dimensional
algebras (Seattle 1997)}, 201-213,
Contemp. Math. 229, Amer. Math. Soc.

 \bibitem{Keller} B.~Keller, \emph{Invariance and localization for 
cyclic homology of DG algebras},
J.~Pure~Appl.~Alg.~\textbf{123} (1998), 223-273

\bibitem{K1}B.~K\"{u}lshammer, \emph{Bemerkungen  \"{u}ber die 
Gruppenalgebra als
symmetrische Algebra I, II, III, IV},   J. Algebra \textbf{72} (1981),
1-7;  J.~Algebra~\textbf{75}   (1982), 59-69;  J.~Algebra~\textbf{88}
(1984), 279-291; J.~Algebra~\textbf{93} (1985), 310-323

 \bibitem{K2} B.~K\"{u}lshammer, \emph{Group-theoretical descriptions 
of ring theoretical invariants of group algebras}, Progress in Math.
\textbf{95}, Birkh\"auser (1991), 425-441

 \bibitem{Nav} G.~Navarro, \emph{Characters and Blocks of Finite Groups},
London Mathematical Society Lecture Notes Series 250, Cambridge
University Press, Cambridge 1998

\bibitem{PS}I.~Piatetski-Shapiro, \emph{Complex Representations of $GL(2,K)$ 
for Finite Fields}, Contemporary Mathematics \textbf{16}, American Mathematical 
Society, 1983

\bibitem{P}L.~Puig, \emph{Nilpotent blocks and their source algebras}, 
Invent.~Math.~\textbf{93} (1988), 77-116

 \bibitem{R1} J.~Rickard, \emph{Derived categories and stable equivalences},  
J.~Pure~Appl.~Alg. \textbf{61 }(1989) 303-317

\bibitem{R2} J.~Rickard, \emph{Morita theory for derived categories},   
J.~London~Math.~Soc.~ \textbf{39} (1989), 436-456

\bibitem{R3} J.~Rickard,  \emph{Derived equivalences as derived functors},  
J.~London~Math.~Soc.~\textbf{43} (1991), 37-48

\bibitem{TT} R.~W.~Thomason and T.~Trobaugh, \emph{Higher algebraic K-theory 
of schemes and of derived categories}, Grothendieck Festschrift III, Progress
in Math. \textbf{86}, Birkh\"{a}user (1988), 248-435

 \bibitem{Z1} A.~Zimmermann, \emph{Invariance of generalized Reynolds ideals 
under derived equivalences},  Mathematical Proceedings of the Royal Irish 
Academy \textbf{107A} (1) (2007) 1-9

\bibitem{Z2}  A.~Zimmermann, \emph{Fine Hochschild invariants of derived
categories for symmetric algebras},  J.~Algebra~  \textbf{308} (2007)
350-367

\bibitem{Z3}  A.~Zimmermann, \emph{Hochschild homology invariants of K\"ulshammer 
type of derived categories}, preprint (2007)
avaible at  http://www.mathinfo.u-picardie.fr/alex/alexpapers.html


\end{thebibliography}
\end{document}